\documentclass[12pt, reqno]{amsart}
\usepackage{}
\usepackage{amsmath, amstext, amsbsy, amssymb, amscd}
\usepackage{amsmath}

\usepackage{bbm}
\usepackage{dsfont}
\usepackage{yfonts}
\usepackage{amsmath}
\usepackage{amssymb}
\usepackage{amsfonts}
\usepackage{mathrsfs}
\usepackage{amsxtra}
\usepackage{amscd}
\usepackage{amsthm}
\usepackage{amsfonts}
\usepackage{amssymb}
\usepackage{eucal}
\usepackage{color}
\usepackage{mathrsfs}
\usepackage[all]{xy}
\usepackage[CJKbookmarks=true]{hyperref}
\usepackage{soul}
\usepackage{shorttoc}

{\vskip-\lastskip\medskip
  \noindent
  {\em #1.}\enspace
  }%
{\qed\par\medskip
  }
\title{Convolution algebra  of diagram automorphism fixed quiver variety}

\newtheorem{theorem}{Theorem}
\newtheorem{example}{Example}

\newenvironment{psmallmatrix}
  {\left(\begin{smallmatrix}}
  {\end{smallmatrix}\right)}
\newcommand{\BB}{B^{\prime}}
\newcommand{\II}{I^{\prime}}
\newcommand{\JJ}{J^{\prime}}

\newcommand{\cM}{\mathfrak{M}}

\usepackage{tikz}

\begin{document}

\author[Z. Dong, H. Ma ]{Zhijie Dong, Haitao Ma }
\address{Institute for Advanced Study in Mathematics of HIT, Harbin, 150001, China}
\address{College of mathematics science, Harbin Engineering University, Harbin, 150001, China.}
\email{dongmouren@gmail.com (Zhijie Dong)}
\email{hmamath@hrbeu.edu.cn (Haitao Ma)}

\maketitle
\begin{abstract}
We study the convolution algebra $H_{*}(Z^{\theta}_{W})$ of homology on diagram automorphism fixed point quiver variety and prove that there exists an algebra homomorphism from the universal enveloping algebra of the diagram automorphism fixed algebra of the split quiver to $H_{*}(Z^{\theta}_{W})$.
\end{abstract}

\section{Introduction}

For a quiver $Q$, Nakajima \cite{nakajima1994instantons,nakajima1998quiver} attached to it 
varieties $\cM(V,W)$ and $\mathfrak{L}(V,W)$ which encode representation meaning of 
Lie algebra $\mathfrak{g}$ associated to $Q$. Denote by $Z_W$ the Steinberg type variety. It is proved that there is a Lie algebra action on the sum of top degree homology groups on $\mathfrak{L}(V,W)$ over $V$ through $Z_W$.
 $$
\mathcal{U}(\mathfrak{g})\rightarrow H^{top}_{*}(Z_W) \curvearrowright \bigoplus_{V}H^{top}_{*}(\mathfrak{L}(V,W)).
$$
If we add an automorphism to $\cM(V,W)$, what is the algebra to replace $\mathcal{U}(\mathbf{g})$?
What is the case for the total homology and $K$-theory version?

In the paper \cite{li2019quiver}, it is conjectured that for a certain type automorphism $\sigma$ on quiver variety, 
there is a corresponding automorphism on the Lie algebra $\mathfrak{g}$, which we also denote by $\sigma$, such that
$$ 
\mathcal{U}(\mathfrak{g}^\sigma)\rightarrow H^{top}_{*}(Z_{W}^{\sigma}). $$ 

In the paper \cite{fan2019equivariant}, the $K$-theory version for a special case of $\sigma$-quiver variety, the flag variety of type B/C, was studied. This type of automorphism $\sigma$ corresponds to symmetric 
pairs. In the paper\cite{ZM}, the authors also studied the total homology version.
 It is expected that the quantum affine algebra (Yangian) for symmetric 
pairs maps to the $K$-theory (homology) convolution algebra on $Z_{W}^{\sigma}$. 
The difficulty for proving this is that the quantum affine algebra (Yangian) for symmetric 
pairs does not have Drinfeld new realization yet.

Compared to this case,
there is another automorphism on the Lie algebra, the diagram
automorphism, which was studied extensively \cite{kac1990infinite,carter2005lie}. The corresponding affine version (twisted affine algebra) and quantization (twisted quantum group and Yangian) have Drinfeld new realizations.  
This motivates us to study the following question.

What is the geometry that realizes the
diagram automorphism fixed point Lie algebra and its cousins (affine version, quantization, etc)?
A natural guess of the variety is the diagram automorphism $\theta$ fixed point subvariety of quiver variety.
In \cite{henderson2014diagram}, it is shown that $\cM^{\theta}(V,W)$ for quiver $Q$ is the disjoint union of $\cM(V^{\prime},W^{\prime})$ over certain $V^{\prime}$ of the split-quotient quiver $s(Q)$. We add prime when we refer to objects associated to the split quotient quiver if not particularly mentioned.
 
We have [Theorem 6]$$
\mathcal{U}(\mathfrak{g}_{s(Q)}) \xrightarrow[]{} H_{*}((Z_W)_{s(Q)}) \xrightarrow[]{} H_{*}((Z_{W})^{\theta}_{Q}).$$
 Moreover, the quiver $s(Q)$ has a natural diagram automorphism $a^{\prime}$ and we proved that the enveloping algebra of the $a^{\prime}$ fixed point subalgebra of $\mathfrak{g}_{s(Q)}$ maps to the subalgebra generated by the $\theta$-Hecke correspondences
of $\cM(V,W)^{\theta}$.

In particular, when $Q$ is of (affine) type $A$ (resp. $D$), the split-quotient quiver $s(Q)$ is (affine) type $D$ (resp. $A$) and $\mathfrak{g}^{a^{\prime}}_{s(Q)}$ is of (affine) type $B$ (resp. $C$).

We want to understand the module structure of $\bigoplus_{V} H^{top}_{*}(\mathfrak{L}(V,W))^{\theta}$ as $\mathfrak{g}^{a^{\prime}}_{s(Q)}$-algebra. As an $\mathfrak{g}_{s(Q)}$ algebra, It is the highest weight module of highest weight $\mathbf{w}^{\prime}$.
So this is purely a branching problem. Unfortunately,  since our variety 
is not new, the geometry does not help us with this problem.

\section{Diagram automorphism quiver variety}
We recall the definition and the main theory of the diagram automorphism quiver variety in \cite{henderson2014diagram}.
Let $Q=\{I,E\}$ be a quiver, 
where $I$ is the set of vertices and $E$ is the set of edges.

For an edge $a=i\xrightarrow[]{}j$, define $s(a)=i,t(a)=j$, let $a^*=j\xrightarrow[]{}i$ and let
  $E^*=\{a^* | a\in E\}$. Define $\epsilon(a)=1 $ when $a\in E$, $\epsilon(a)=-1$ when $a\in E^*$.
  Denote the doubled quiver by $\bar{Q}$.
Define the framed double quiver by $fQ$  doubling the vertex set $I$ and add arrows between $(i,0)$ and $(i,1)$ for $i\in I$.

A $\Pi_{fQ}-$module is the data of two  $I$-graded vector spaces $\bigoplus_{i\in I} V_i$ and $\bigoplus_{i\in I}W_i$ and linear maps $ V_{s(h)} \xrightarrow[]{B_h} V_{t(h)},  V_i\xrightarrow[]{I_i} W_i,  W_i\xrightarrow[]{J_i} V_i$ for each $h\in E\bigsqcup E^{*},i\in I$ satisfying $\sum\limits_{h|s(h)=i}B_{\overline{h}}B_{h}+I_iJ_i=0$ for any $i\in I$.
Let $G_V=\Pi_{i\in I} GL_{V_i}$ act on the set of all $\Pi_{fQ}-$module by conjugation.
For a $\Pi_{fQ}-$module $M=(B,I,J)$, we say $M$ is stable if there is no nonzero $k\overline{Q}$-submodule contained in $kerJ$.
For $I$-graded vector space $V$ (resp. $W$) of fixed dimension vectors $\mathbf{v}$ (resp. $\mathbf{w}$), denote the set of all  $\Pi_{fQ}$-modules by $M(V,W)$ and the set of all stable modules by $M^{s}(V,W)$. By  \cite[Lemma 3.10]{nakajima1998quiver}, the stabilizer for stable module is identity. 
Quiver variety $\cM(V,W)$ is the geometric quotient $M^{s}(V,W)/G_V$
and $\cM_0(V,W)$ is the GIT quotient whose geometric point is the close $G_V$-orbit.  There is a map  $\cM(V,W)\xrightarrow{\pi} \cM_0(V,W) $ sending a stable module to its semisimplification, which is the unique close orbit in its $G_V$-orbit closure.
An automorphism $a$ on quiver $Q=(I,E)$ is an automorphism of set $I$ and $E$, such that $t,s$ commute with $a$.
We say $a$ is admissible if vertices in each orbit are disconnected.
\begin{example}
Let $Q$ be of type $A_{2n-1}$ and $a(i)=2n-i$. This is admissible. Let $Q$ be $A_{2n}$ and $a(i)=2n-i.$ This is not admissible. 
Let $Q$ be of type $D_n$ and $a(i)=i$ for 
$i\leq n-1$, $a(n-1)=n,\ a(n)=n-1.$ This is admissible.
\end{example}
We now define the action $\theta$ on the quiver variety $\cM(V,W)$ (resp. $\cM_{0}(V,W)$). 

First, we define the action $\theta$ on $M(V,W)$.  
 %
For this, we introduce the split quotient quiver $s(Q)$. 
Since we always double the quiver when define quiver variety so we only have to define $s(Q)$ as a diagram.
Given $Q=(I,E)$, let $d_h$, $d_i$ be the size of orbit $\langle a \rangle h$, $\langle a \rangle i$, where $\langle a \rangle$ is the group generated by $a$. Let $n$ be the least common multiple of all $d_h$, $d_i$ where $h\in E$, $ i\in I$ and $e_i=n/d_i, e_h=n/d_h$.
Define the quotient quiver $q(Q)=(q(I),q(E))$ where the set of vertices $q(I)$ consist of  $\langle a \rangle$-orbits of $I$ and
$q(E)$ consist of $\langle a \rangle$ orbits of $E$.
Define the split quotient quiver consisting of vertices  $I^{\prime}=\{(\bar{i},j/e_i), i\in q(I), 1\leq j\leq e_i\}$
and edges between $(\bar{i_1},j_1/e_{i_1})$ and $(\bar{i_2},j_2/e_{i_2})$ 
 whenever
 $e_{i_2}j_1/e_{i_1}=e_{i_1}j_2/e_{i_2}\ (mod\ 1)$. 
The split quotient quiver $s(Q)$ also has an automorphism $a^{\prime}$ where  $a^{\prime}(\bar{i},j/e_i)=(\bar{i},(j+1)/e_i\ (mod \ 1))$
and $s(s(Q))=Q$.

To define $\theta$,
we need isomorphisms $V_i\xrightarrow{\phi_i}V_{a(i)}$ and $W_i\xrightarrow{\sigma_i}W_{a(i)}$.
Since we finally want to define it on the $G_V$-orbit we can always assume $\phi_i=id$ but
we only pose the weaker condition that $( \sigma_{a^{d_i-1}(i)}\cdots \sigma_i)^{e_{i}}=id$.

Then the induced map is $\theta(B_{a(h)})=B_{h}$ and $\theta(J_{a(i)})=\sigma_{i}J_{i},\  
\theta(I_{a(i)})=I_i\sigma_i$.
The morphism $\theta$ clearly descents to a morphism on $\cM(V,W)$, $\cM_0(V,W)$ and satisfies $\theta^{n}=id$.
It is clear to define $\theta$, the dimension of $V$ and $W$ satisfy $\mathbf{v}_i=\mathbf{v}_{a(i)}$ and $\mathbf{w}_i=\mathbf{w}_{a(i)}$, which we will implicitly assume.
Denote $\cM(V,W)^{\theta}$(resp. $\cM_{0}(V,W)^{\theta}$) the fixed point subvariety of $\cM(V,W)$ (resp. $\cM_{0}(V,W)$).
Here are some examples.
If the original quiver is\\
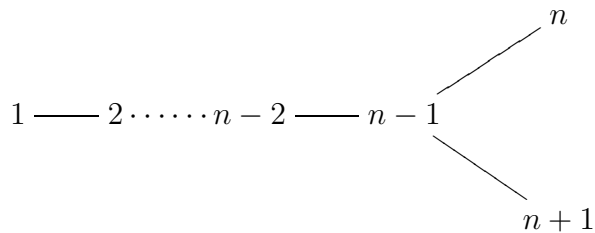
\centerline{
\xymatrix{ &&& n \\
1 \ar@{-}[r]&2 \cdots \cdots n-2 \ar@{-}[r]&   n-1 \ar@{-}[ru]
\ar@{-}[rd]\\ &&& n+1
}}\\
where $a$ fixed i for $i\leq n-1$ and $a(n-1)=n,\ a(n)=n-1$. The split quotient of $Q$ is 
\\
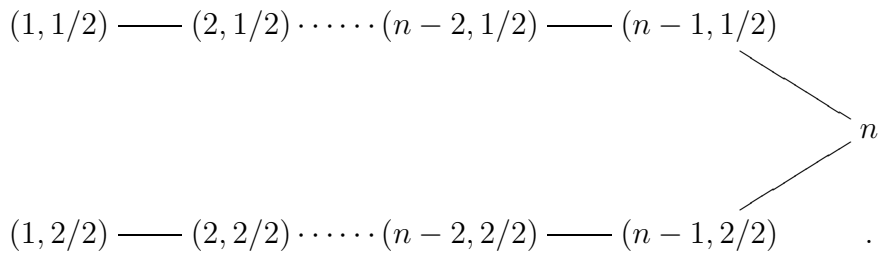
\centerline{
\xymatrix{  (1,1/2)\ar@{-}[r]&(2,1/2) \cdots \cdots (n-2,1/2) \ar@{-}[r] &(n-1,1/2) \ar@{-}[rd]&\\
&&& n\\
(1,2/2)\ar@{-}[r]&(2,2/2) \cdots \cdots(n-2,2/2) \ar@{-}[r] & (n-1,2/2)\ar@{-}[ru]&
.}}\\
We see that $Q$ and $s(Q)$ are of type $D_{n+1}$ and $A_{2n-1}$.


\begin{theorem}\cite[theorem 3.9]{henderson2014diagram} 

Given $\sigma_{a^{d_i-1}(i)}\cdots \sigma_i$ action on $W_i$, we have eigenspace decomposition $W_i=\oplus_{1\leq j\leq e_i} (W_i)_{j}$. We get an
$I^{\prime}$-graded space $W^{\prime}$ with $W_{\Bar{i},j}= (W_i)_{j}$ where $i\in I$ is any lift of $\Bar{i}$ in the quotient quiver.

Given $I^{\prime}$-graded space $V^{\prime}$ of $s(Q)$, we get a $I$-graded space, which we denote by $p(V^{\prime})$, such that $p(V^{\prime})_i=\oplus_{1\leq j \leq e_i}(V_{\bar{i},j/e_i})$. 

There is an isomorphism $\phi$:  
$$\bigsqcup \limits_{\mathbf{v}^{\prime}|p(\mathbf{v}^{\prime})=\mathbf{v}} \cM(V^{\prime},W^{\prime})\xrightarrow[]{\phi}\cM(V,W)^{\theta},$$
where the union is over all dimension vectors $\mathbf{v^{\prime}}$ such that
$p(\mathbf{v^{\prime}})=\mathbf{v}$.
\end{theorem}

\section{Convolution}
We first recall convolution of quiver variety case.
We have $\cM(V,W)\xrightarrow[]{\pi} \cM_{0}(V,W)$ in each component. We want to take union over all $V$.
To write it as $X\times_{Y} X$ where $X$ is the union of $\cM(V,W)$ over $V$, we have to take $Y$ as some limit of $\cM_{0}(V,W)$. 
To avoid this, we write it as union of components where each of them is of Steinberg type.
For fixed $W$, denote the variety by $Z_W$ the disjoint union of $Z(V_1,V_2)$, where $Z(V_1,V_2)$ is the fiber product of 
$$\cM(V_1,W) \times_{\cM_{0}(V_1+V_2,W)} \cM(V_2,W).$$
The maps of both sides  are composition of $\pi$ and the map $\cM_0(V_i,W)\hookrightarrow \cM_0(V_1+V_2,W)$, where the latter maps take direct sum with the zero module.

We use the notation $H_{*}$ for Borel-Moore homology
and when we say homology we mean Borel-Moore homology.
We summarize how $H_{*}(Z)$ is endowed with an algebraic structure, see \cite{chriss2009representation} for details.
Let $X_1,\  X_2,\  X_3$ be varieties. 
Let $Z_{12}$ be subvariety in $X_1\times X_2$ and $Z_{23}$ be subvariety in $X_2\times X_3$.
Denote the projection $X_1\times X_2 \times X_3  \xrightarrow{} X_1\times X_2$ by $\pi_{12}$ and $\pi_{23}$, $\pi_{13}$ similarly.
Given two homology classes $z_{12}\in H_{*}(Z_{12})$ and $z_{23}\in H_{*}(Z_{23})$, define the convolution product \\
$$z_{12}\bullet z_{23}=(\pi_{13})_{*}(\pi^{*}_{12}(z_{12})\cap \pi^{*}_{23}(z_{23})),$$
where lower star means pushforward which is defined for proper map and upper star means pullback which is defined for bundle map.
Define the set theoretic image of  $\pi_{12}(Z_{12})\cap \pi_{23}(Z_{23})$ under $\pi_{23}$  by $Z_{13}$.
We see that $z_{12}*z_{23}\in H_{*}(Z_{13})$ hence we have
$H_{*}(Z_{12})\times H_{*}(Z_{23})\xrightarrow{\bullet}H_{*}(Z_{13})$ and it can be shown that this convolution product is associative.
Given a proper map $X\xrightarrow{\pi} Y$, 
we set $X_1=X_2=X_3=X$ and $Z_{12}=Z_{23}=Z=X\times_{Y} X$.
We have $Z_{13}=Z$ and $H_{*}(Z)\times H_{*}(Z) \xrightarrow{\bullet} H_{*}(Z)$ so $H_{*}(Z)$ has an associative algebra structure.
If we set $X_1=X_2=X, \ X_3=pt$ , $Z_{12}=Z=X\times_{Y} X$, $Z_{23}=X\times_{Y} pt$, where $pt$ maps to $y\in Y$ in the second map of the fiber product, we have $Z_{13}=Z_{23}=\pi^{-1}(y)$ and the convolution $H_{*}(Z)\times H_{*}(\pi^{-1}(y)) \xrightarrow{\bullet} H_{*}(\pi^{-1}(y))$ gives $H_{*}(\pi^{-1}(y))$ a $H_{*}(Z)$ module.

The convolution product preserves degree of half dimension.
If $X\times_{Y} X$ is half dimension of $X\times X$ and $\pi^{-1}(y)$ is half dimension of $X$, $H^{top}_{*}(X\times_{Y} X)$ is an subalgebra of $H_{*}(X\times_{Y} X) $ and $H^{top}_{*}(\pi^{-1}(y))$ is a $H^{top}_{*}(X\times_{Y} X)$-module.

In this case $Z_W$ is disjoint union of $Z(V_1,V_2)$ and we can consider the  product in each component and half dimension means in each component.  It is proved in \cite{nakajima1998quiver} that $Z(V_1,V_2)$ is equidimensional of dimension half of dim$\cM(V_1,W)$+ dim$\cM(V_2,W)$, 
so  $\bigsqcup_{V} H^{top}_{*}(\mathfrak{L}(V,W))$ is a
$H^{top}_{*}(Z_W)$ module. In \cite{nakajima1998quiver}, define Hecke correspondence $\mathcal{B}_i$ as the variety of 
pairs of stable modules $((\BB,\II,\JJ),(B,I,J))$ quotient by the $G_V$ such that 
$(\BB,\II,\JJ)$ is a framed submodule of $(B,I,J)$ with dimension vector less than 1 in the $i$-th graded piece. 
It can be identified with the closed subvariety in $Z(V^{\prime},V)$ consisting of $([\BB,\II,\JJ],[B,I,J])$ such that there exists injective map $ (\BB,\II,\JJ)\xrightarrow[]{\xi}(B,I,J)$.
In \cite{nakajima1998quiver}, it is showed that $\mathcal{B}_i$ is one of the irreducible components of $Z(V^{\prime},V)$.
It is also proved that there is map
\begin{align}
 \mathcal{U}(\mathfrak{g})\xrightarrow{}  H^{top}_{*}(Z_W),
\end{align}
which sends Chevalley generators $E_i$ to the homology class of the
Hecke correspondences $\mathcal{B}_i$.

We now bring everything to $\theta$ case.

\begin{theorem}\cite[prop 3.19]{henderson2014diagram}
The diagram commutes\\
\centerline{
\xymatrix{ \bigsqcup\cM(V^{\prime},W^{\prime}) \ar[r]^{\psi} \ar[d]^{\pi^{\prime}} & \cM(V,W)^{\theta}\ar[d]^{\pi^{\theta}} \\ \bigsqcup\cM_{0}(V^{\prime},W^{\prime}) \ar[r]^{\psi_{0}} & \cM_{0}(V,W)^{\theta}.
}}
\end{theorem}
\begin{proof}
We give a proof for later use.
Take $M\in \cM(V^{\prime},W^{\prime})$, the image $\pi^{\prime}(M)$ is the semisimplification of $M$ which means that it has the same simple factor as $M$ and $\pi^{\prime}(M)$ is semisimple.
We have $\psi(M)$ and $\psi_{0}(\pi^{\prime}(M))$ have the simple factors. We need to prove that $\psi_{0}(\pi^{\prime}(M))$ is semisimple.
It suffices to prove that if $M$ is simple, $\psi_{0}(M)$ is also simple.
Without loss of generality, suppose $\psi(M)$ is the following module\\$$
\xymatrixcolsep{10pc}\xymatrix{
W \ar@<.5ex>[d]^{I} & W_{+}\oplus W_{-} \ar[d]^{\begin{psmallmatrix}
I^{+}&0 \\
0&I^{-} \end{psmallmatrix}} & W\ar@<.5ex>[d]^{I}\\
V \ar@<.5ex>[u]^{J} \ar@<.5ex>[r]^{\begin{psmallmatrix}
A^{+}\\A^{-}
\end{psmallmatrix} 
} & V_{+}\oplus V_{-} \ar@<0.5ex>[l]^{ 
(B^{+},B^{-})} \ar@<0.5ex>[u]^{\begin{psmallmatrix}
2J^{+}&0 \\
0&2J^{-} 
\end{psmallmatrix} 
}\ar[r]^{(B^{+},-B^{-})}  & V\ar@<0.5ex> [l]^{\begin{psmallmatrix}
A^{+}\\-A^{-}
\end{psmallmatrix}} \ar@<0.5ex>[u]^{J}.}$$
If it is not simple, there exists a module\\
$$\xymatrixcolsep{10pc}\xymatrix{
W \ar@<.5ex>[d]^{I} & W_{+}\oplus W_{-} \ar[d]^{\begin{psmallmatrix}
I^{+}&0 \\
0&I^{-} \end{psmallmatrix}} & W\ar@<.5ex>[d]^{I}\\
V^{\prime} \ar@<.5ex>[u]^{J} \ar@<.5ex>[r]^{\begin{psmallmatrix}
A^{+}\\A^{-}
\end{psmallmatrix} 
} & V_{+}^{\prime}\oplus V_{-}^{\prime} \ar@<0.5ex>[l]^{ 
(B^{+},B^{-})} \ar@<0.5ex>[u]^{\begin{psmallmatrix}
2J^{+}&0 \\
0&2J^{-} 
\end{psmallmatrix} 
}\ar[r]^{(B^{+},-B^{-})}  & V^{\prime \prime}\ar@<0.5ex> [l]^{\begin{psmallmatrix}
A^{+}\\-A^{-}
\end{psmallmatrix}} \ar@<0.5ex>[u]^{J}.}$$
Take $v\in V_{+}^{\prime}$, by LHS of the quiver we have $B^{+}v\in V^{\prime}$ and by RHS we have $B^{+}v\in V^{\prime \prime}$, so $B^{+}v\in V^{\prime}\cap V^{\prime \prime}$.
Similarly for any $v\in V_{-}^{\prime}$, we have $B^{-}v\in V^{\prime}\cap V^{\prime \prime}$.
So for any $v\in V^{\prime}_{+}\oplus V^{\prime}_{-}, (B^{+},B^{-})v \in V^{\prime}\cap V^{\prime \prime}$.
So if we replace $V^{\prime}$ and ${V^{\prime\prime}}$ by 
$V^{\prime}\cap V^{\prime \prime}$ it is still a submodule. It is an image of some module $M^{\prime}$ under $\psi_{0}$. Since $M^{\prime}$ is a submodule of simple module $M$, hence $\psi_{0}(M)$ is simple. Therefore the image under $\pi^{\theta}$ of $\cM(V,W)^{\theta}$ is in $\cM_{0}(V,W)^{\theta}$.

\end{proof}
Denote the image of $\pi^{\theta}$ by $\cM_{1}(V,W)^{\theta}$.

\begin{theorem}
The map $\cM(V,W)^{\theta}\xrightarrow[]{\pi^{\theta}} \cM_0(V,W)^{\theta}$ is not surjective.

\end{theorem}

\begin{proof}

We say $M$ is $\theta$-stable if $M$ and $\theta(M)$ is in the same $G_V$-orbit, which means $\theta(M)=gM$ for some $g\in G_V$. We will call $g$ the transition matrix of $M$, though it is a product of matrices $g_i \in GL(V_i),\  i\in I$.
Take any simple module $M_1$ that is not $\theta$-stable.
For any $M$ of dimension $(\mathbf{v},\mathbf{w})$, $\theta(M)$ is of dimension $(v_{a(i)})_{i\in I},(w_{a(i)})_{i\in I}$.
Take any $g\in \Pi GL(V_{a(i)})$ and 
consider a module $M=M_1\oplus g\theta M_1$.
We have $\theta(M)=\theta(M_1)\oplus \theta(g\theta(M_1))=\theta(g\theta(M_1))\oplus \theta(M_1)$.

We want to show that $M$ is $\theta$-stable.
We show that for any $g\in \Pi GL(V_{a(i)})$,
there exists some $g^{*}\in \Pi GL(V_i)$,
such that $\theta g\theta= g^{*}.$
This is direct computation as follows. The $(I,J)$ part of $M$ is trivial so we only compute $B$ part.
The arrow between $i$ and $i+1$ is
\xymatrix{
V_i \ar[r]^{B_{i,i+1}} &V_{i+1}
},
after applying $\theta$, it is 
\xymatrixcolsep{5pc}\xymatrix{
V_{a^{-1}(i)} \ar[r]^{B_{a^{-1}(i),a^{-1}(i+1)}}& V_{a^{-1}(i+1)}
}.
Note this is still between vertex $i$ and $i+1$. After applying $g$, \\it is
\xymatrixcolsep{10pc}\xymatrix{
V_{a^{-1}(i)} \ar[r]^{(g_{i+1}) B_{a^{-1}(i),a^{-1}(i+1)}(g_{i})^{-1}}& V_{a^{-1}(i+1)}
} and after applying $\theta$, what we get between position $i$ and $i+1$ is what previously between $a(i)$ and $a(i+1)$, which is
\xymatrixcolsep{10pc}\xymatrix{
V_{i} \ar[r]^{(g_{a(i+1)}) B_{i,i+1}(g_{a(i)})^{-1}}& V_{i+1}
}.
From the above computation, we see that  $g^{*}_{i}=g_{a(i)}$.
Hence $$\theta(M)=g^{*} M_1 \oplus \theta(M_1)=\begin{pmatrix} g^{*}&\\& g^{-1} \end{pmatrix} (M_1,g\theta M_1)=\begin{pmatrix} g^{*}&\\& g^{-1} \end{pmatrix} (M).$$
This means $M$ is $\theta$-stable and the transition matrix is $\begin{pmatrix} g^{*}&\\& g^{-1} \end{pmatrix}$ and since $M$ is simple the transition matrix is unique at the block diagonal part.
If we choose a representative $h(M), h\in G_V$ in the $G_V$-orbit of $M$,
$\theta (h(M))= h^{*}\theta(M)=h^{*} g(M)=
h^{*} g h^{-1} (h(M)).$
Without loss of generality, we assume there is a $a$-fixed vertex $i$ with $e_i=2$.
This means the transition matrix $g^{\prime}$ of $h(M)$ is $h^{*} g h^{-1}$, in particular,
on vertex $i$, the transition matrix $g^{\prime}_{i}$ is a conjugation of $g_{i}$.
We know that the module in the image of $\Psi_{0}$ has a representative such that the transition matrix at the fixed vertex $i$ is $\begin{pmatrix} I{+}&\\& -I_{-} \end{pmatrix}$, so if we choose $g_i$ not having eigenvalues all being 1 or -1, $M$ can not be a representative of the image of $\pi^{\theta}$.
\end{proof}
Although $\pi^{\theta}$ is not surjective,
the fiber product $$\cM(V,W)^{\theta}\times_{\cM_{1}(V,W)^{\theta}} \cM(V,W)^{\theta}$$ is the same as $$\cM(V,W)^{\theta} \times_{\cM_{0}(V,W)^{\theta}} \cM(V,W)^{\theta}.$$

Moreover we have an proper inclusion $$\bigsqcup \limits_{p(V^{\prime}_{i})=V_i,\  i=1,2} \cM(V_{1}^{\prime},W^{\prime})\times_{\cM_{0}(V_{1}^{\prime}+V_{2}^{\prime},W^{\prime})}  \cM(V_{2}^{\prime},W^{\prime}) \hookrightarrow
 \cM(V_{1},W)^{\theta}\times_{\cM_{1}(V_{1}+V_{2},W)^{\theta} }
 \cM(V_{2},W)^{\theta}.$$
 We denote the latter by $Z(V_1,V_2)^{\theta}$ and the inclusion by $\Pi_{V_1,V_2}$.
Then
 \begin{align}
 Z_{W^{\prime}}=\bigsqcup \limits_{V_{1}^{\prime},V_{2}^{\prime}} Z(V_{1}^{\prime},V_{2}^{\prime})\cong\\\bigsqcup \limits_{V_{1},V_{2}}
 \bigsqcup \limits_{V_{1}^{\prime},V_{2}^{\prime}|p(V_{1}^{\prime})= V_1,p(V_{2}^{\prime})=V_2} Z(V_{1}^{\prime},V_{2}^{\prime})\xrightarrow[]{\Pi_{V_1,V_2}}\\
\bigsqcup_{V_1,V_2} Z(V_1,V_2)^{\theta}.
\end{align}
 
 The first one is what we already understood. But 
 we could ask what is the subalgebra generated by Hecke correspondences in $H_{*}(\bigsqcup Z(V_1,V_2)^{\theta})$.
 
 \section {$\theta$-Hecke correspondence}
 Define $\mathcal{B}^{\prime}_i$ be the variety of pairs of stable and $\theta$-stable modules $((\BB,\II,\JJ
 ),(B,I,J))$ quotient by $G_V$ such that $(\BB,\II,\JJ)$ 
is a framed submodule of $(B,I,J)$ whose dimension is less than 1 in all $j$-th graded pieces, where $j$ is in $\langle a \rangle$-orbit of $i$. It can be identified with a closed subvariety in $Z(V_1,V_2)^{\theta}$, which we will refer to $\theta$-Hecke correspondence.

\begin{theorem}
The $\theta$-Hecke correspondence $\mathcal{B}_i^{\prime}$ of $Q$ is isomorphic to the disjoint union of Hecke correspondences $\mathcal{B}_j$ of $s(Q)$, where $j$
runs over all vertices that $i$ splits into.

\end{theorem}
\begin{proof}
This follows from the next theorem.
\end{proof}
 
\begin{theorem}
Given a pair of stable and $\theta$-stable module $M^{\prime}\subset M$ such that $\theta(M)=gM$ for $g \in G_V$ and $\theta(M^{\prime})=g^{\prime}M^{\prime}$ for $g^{\prime}\in G_{V}^{\prime}$, where $g$ and $g^{\prime}$ are the transition matrices of $M$ and $M^{\prime}$ . For any $i\in I$, the eigenspace of $g^{\prime}$ on $M^{\prime}$ is a subspace of the eigenspace of $g$ on $M$ of the same eigenvalue.
\end{theorem}
\begin{proof}
We could always assume $M=(B,I,J)$ is of the form of $\phi(M)$, note this $\phi$ is not the map on $\cM(V,W)$ but on $M(V,W)$. See \cite[3.23]{henderson2014diagram} for details. 

Since $V_i\xrightarrow[]{B_{ij}} V_j$ maps eigenspace of $V_i$ under $g_i$  to eigenspace of  $V_j$ under $g_j$, we can reduce the theorem to the vertex that has branching, meaning there exist changes of size on its adjacent orbits.
We see that $(B,I,J)$ is block matrices so we can reduce the theorem to the case where one vertex branches into $m$ vertices.

For brevity of notation, we assume that $m=2$, and the configuration is as below. 
\xymatrixcolsep{10pc}\xymatrix{
W \ar@<.5ex>[d]^{I} & W_{+}\oplus W_{-} \ar[d]^{\begin{psmallmatrix}
I^{+}&0 \\
0&I^{-} \end{psmallmatrix}} & W\ar@<.5ex>[d]^{I}\\
V \ar@<.5ex>[u]^{J} \ar@<.5ex>[r]^{\begin{psmallmatrix}
A^{+}\\A^{-}
\end{psmallmatrix} 
} & V_{+}\oplus V_{-} \ar@<0.5ex>[l]^{ 
(B^{+},B^{-})} \ar@<0.5ex>[u]^{\begin{psmallmatrix}
2J^{+}&0 \\
0&2J^{-} 
\end{psmallmatrix} 
}\ar[r]^{(B^{+},-B^{-})}  & V\ar@<0.5ex> [l]^{\begin{psmallmatrix}
A^{+}\\-A^{-}
\end{psmallmatrix}} \ar@<0.5ex>[u]^{J}}

By \cite[3.23]{henderson2014diagram}, we can assume the module $(B,I,J)$ is given in the above form.
Suppose we have a submodule $(\BB,\II,\JJ)$.
Since it is $\theta$-stable, there exists some $(h,g,f)$ such that $\theta(\BB,\II,\JJ)=(h,g,f)(\BB,\II,\JJ)$. This implies $(hf,g^2,hf)(\BB,\II,\JJ)=(\BB,\II,\JJ)$ and since $(\BB,\II,\JJ)$ is stable we have $g^{2}=id, f=h^{-1}$. 
So $V^{\prime}$ is spanned by all eigenvectors of eigenvalue 1 and -1.
We take any eigenvector $v$ of eigenvalue 1, so $gv=v$ and have 
$$\begin{psmallmatrix}
2J^{+}&0 \\
0&2J^{-} 
\end{psmallmatrix} v=\begin{psmallmatrix}
2J^{+}&0 \\
0&-2J^{-} 
\end{psmallmatrix} gv.$$

We write $v=v_1+v_2$, where $v_1\in V_{+}$ and $v_2\in V_{-}.$
we have $J^{+}v_1+J^{-}v_2=J^{+}v_1-J^{-}v_2$ so $J^{-}v_2$=0.

We have also
\begin{equation} \label{eq1}
\begin{split}
h(B^{+},B^{-})v=(B^{+},-B^{-})gv. \\
 hB^{+}v_1+hB^{-}v_2=B^{+}v_1-B^{-}v_2\\
 hB^{+}v_1-B^{+}v_1=-hB^{-}v_2-B^{-}v_2     .
\end{split}
\end{equation}
Moreover $Jh=J$.
For any vector $w\in V$, $Jhw=Jw$, so $hw-w\in kerJ$.
We take $w$ be the image of $v$ under 
$$\begin{psmallmatrix} 
B^{+} \\
B^{-} \end{psmallmatrix},$$
i.e $w=B^{+}v_1+B^{-}v_2.$
Compute
\begin{align}
hw-w=hB^{+}v_1+hB^{-}v_2-B^{+}v_1-B^{-}v_2\\
=hB^{+}v_1-B^{+}v_1+hB^{-}v_2-B^{-}v_2
\\
=-hB^{-}v_2-B^{-}v_2+hB^{-}v_2-B^{-}v_2\\
=-2B^{-}v_2,
\end{align}
where $(7)\Rightarrow(8)$ is by $(5)$.
So $B^{-}v_2\in kerJ$ and
$(B^{-}v_2, v_2, B^{-}v_2)\subset (V,V_{+}\oplus V_{-},V)$ forms an submodule which is contained in $kerJ$. By stability condition, it is zero.
So $v=v_1\in V_{+}$. Similarly, if $v\in V^{\prime}$ is an eigenvector of eigenvalue -1, we have $v=v_2\in V_{-}$.

The general case is similar, one only need to replace the $(+,-)$ by $m$-roots of unity and for more vertices on the left or right, we just follow the arrows to the end to generate the submodule.

 \end{proof}

\begin{theorem}
There is a map  $$\mathcal{U}(\mathfrak{g}_{s(Q)}^{a^{\prime}})\xrightarrow[]{} H_{*}(Z^{\theta}_{W}),$$
that takes generators $e_{i}^{\prime}$ to the class of $\mathcal{B}^{\prime}_i$.
\end{theorem}
\begin{proof}
Let $e^{\prime}_i$ maps to  the fundamental class of $\mathcal{B}^{\theta}_{i}$.
Combining (1)-(4), we have
$$\mathcal{U}(\mathfrak{g}_{s(Q)}) \xrightarrow[]{} H_{*}((Z_W)_{s(Q)}) \xrightarrow[]{} H_{*}((Z_{W})^{\theta}_{Q}),$$
where $e_{i}$ maps to the fundamental class of $\mathcal{B}_{i}$. Then by Theorem 4,  
 we have $e^{\prime}_i=e_i+\cdots e_{a^{d_{i}-1}}$.
Similarly for $f_i^{\prime}$, therefore the theorem follows.
\end{proof}


\bibliographystyle{alpha}
\bibliography{main.bib}

\end{document}